\documentclass[preprint,12pt]{elsarticle}

\usepackage{graphics}
\usepackage{graphicx}
\usepackage{epsfig}
\usepackage{xcolor,soul}
\usepackage{anysize}
\usepackage{array}
\usepackage[fleqn]{amsmath}
\usepackage{latexsym,amssymb,amscd,amsthm,amsxtra}
\usepackage{subfigure}
\usepackage{tabularx}
\usepackage{enumerate}
\usepackage{multirow}
\usepackage{float}
\usepackage{placeins}
\usepackage{booktabs}
\usepackage{times}
\usepackage{soul,xcolor}
\usepackage{algorithm}
\usepackage{algpseudocode}
\usepackage[permil]{overpic}
\usepackage{tikz}
\usetikzlibrary{shapes.geometric, arrows}

\tikzstyle{startstop} = [rectangle, rounded corners, minimum width=2.5cm, minimum height=0.8cm, text centered, draw=black, fill=red!30, font=\small]
\tikzstyle{process} = [rectangle, minimum width=3.5cm, minimum height=0.8cm, text centered, draw=black, fill=blue!30, font=\small, text width=3.5cm, align=center]
\tikzstyle{arrow} = [thick,->,>=stealth]

\newtheorem{proposition}{Proposition}[section]
\newtheorem{condition}{Condition}

\newtheorem{lemma}{Lemma}[section]
\newtheorem{definition}{Definition}[section]
\usepackage[most]{tcolorbox}

\usepackage{arydshln}

\def\R{{\mathbb{R}}}

\begin{document}

\begin{frontmatter}

 \title{An Adaptive Lagrangian B-Spline Framework for Point Cloud Manifold Evolution}
 
%\tnotetext[label1]{}
\author{Muhammad Ammad\corref{cor1}\fnref{label1}}
\ead{21481199@life.hkbu.edu.hk}
\cortext[cor1]{Corresponding author}

\author{Leevan Ling\fnref{label1}}
\ead{lling@hkbu.edu.hk}
%\cortext[cor2]{}

\affiliation[label1]{organization={Department of Mathematics, Hong Kong Baptist University},
            city={Kowloon Tong, Kowloon},
            country={Hong Kong}}

\begin{abstract}
We extend our recent curve-evolution framework based on localized B-spline interpolation to present an adaptive Lagrangian framework for the geometric evolution of point-cloud data representing smooth, codimension-one surfaces in $\R^3$. The method constructs overlapping, localized tensor-product B-spline patches, enabling direct, meshless surface evolution from discrete samples. Within each patch, the differentiable B-spline representation yields analytic, high-order estimates of intrinsic geometric invariants, supporting curvature-driven and geometry-coupled flows. The organization of control points facilitates coherent updates of both surface samples and spline coefficients under intrinsic velocity fields. A conditioning-aware formulation of the local interpolation system, combined with a Gauss–Seidel refinement of control points, maintains interpolation quality throughout the evolution. Adaptive knot insertion and point redistribution, guided by geometric error indicators and local sampling density, preserve surface resolution and regularity during deformation. Numerical experiments demonstrate efficient and accurate reproduction of surface evolution phenomena, including mean-curvature flow, anisotropic deformations, and coupled surface–field dynamics, establishing localized B-spline methods as precise and versatile tools for dynamic manifold approximation.

\bigskip
\noindent \textbf{Research highlights}:
\begin{itemize}
    \item  Localized tensor-product B-spline patches enable flexible, efficient interpolation on evolving point clouds without global meshing.
    \item Interpolant/control points are updated directly at each time step, avoiding reinterpolation.
    \item Control-point coefficients have a clear geometric meaning, allowing in-place updates under curvature-driven motion.
    \item The conditioning-aware formulation for local interpolation preserves the quality of the interpolation under surface changes.
    \item Adaptive point management (insertion/removal) maintains appropriate sampling and surface resolution.
    \item The numerical results confirm a reliable simulation of curvature-driven flows and coupled surface–scalar interactions.
\end{itemize}

\end{abstract}

\begin{keyword}
B-spline interpolation \sep Lagrangian approach \sep Curvature flows \sep Adaptive refinement \sep Reaction-Diffusion systems

\end{keyword}

\end{frontmatter}

\section{Introduction}\label{sec:intro}
The study of manifold evolution, particularly in higher-dimensional settings, occupies a central role in geometric modeling and computational geometry. Foundational work has demonstrated how geometric shapes evolve dynamically under the influence of intrinsic properties, such as curvature or external forces, while adhering to constraints that preserve key geometric quantities such as area, volume, and smoothness \cite{dolcetta2002area, ma2014non, tsai2015length}. These evolutionary processes are not only of theoretical interest but are also fundamental to a wide range of scientific and engineering applications, where the behavior of evolving surfaces critically influences phenomena in fields ranging from biological systems to materials science.

In modeling such evolutionary behavior, two primary computational strategies have emerged: the Eulerian and Lagrangian approaches. The Eulerian approach conceptualizes the evolving surface as a level set of a higher-dimensional scalar function, enabling the handling of topological changes and implicit geometries tracking \cite{Eulerian+bates2009, Eulerian+li2011, Eulerian+mikula20114, Eulerian+mikula2015, Eulerian+ren2023}. In contrast, the Lagrangian approach directly evolves the surface by tracing the trajectories of its constituent points \cite{Lagrangian+barrett2008, Lagrangian+torres2019, Lagrangian+hu2022, Lagrangian+duan2024, Lagrangian+bai2024}. This paper adopts the latter strategy, as it provides enhanced precision in tracking geometric features and facilitates localized control over surface deformation, attributes that are particularly advantageous when analyzing curvature-driven flows and complex surface behavior.

A prime example of curvature-driven evolution is the mean curvature flow, which can produce minimal surfaces and illustrates the fundamental role of curvature in geometric dynamics \cite{minarcik2022minimal}. The behavior of more intricate surfaces, such as hyperbolic paraboloids or dynamically evolving spherical caps, can also be effectively analyzed through their geometric descriptors, including velocity fields, normals, and curvature-related quantities \cite{hyper+wang2023, spherical+abels2015, spherical+epshteyn2021}. A central challenge in such simulations is to maintain geometric integrity over time. Without appropriate correction mechanisms, curvature-induced motions can introduce distortions, particularly in regions of high curvature or anisotropic geometry. Recent research has incorporated additional geometric or physical constraints into the evolution process to mitigate such issues, which help preserve structural features and stabilize the simulation \cite{ALE+elliott2012, ALE+kovacs2024}.

These constraints are mathematically significant and crucial for accurately modeling real-world phenomena. In applications such as fluid dynamics and biology, evolving interfaces must respond sensitively to both internal forces (e.g., curvature) and external stimuli (e.g., pressure, surfactant concentration). For example, the evolution of the curvature-dependent interface plays an essential role in modeling surfactant transport in multiphase flows \cite{multiphase+worner2012, multiphase+garcke2013}, as well as in biological contexts such as cell motility, tissue morphogenesis, and pattern formation \cite{cellmotility+Elliott2012, pattern+barreira2011, pattern+kim2020}. Similarly, in material science, dynamic surface evolution is key to understanding diffusion-driven grain boundary motion, phase transitions, and interfacial interactions \cite{grain+fife2001, fluid+ferziger2019, fluid+joos2023}. These applications demand computational frameworks that are not only accurate and stable, but also adaptable to complex, data-driven surface geometries.

In this context, we propose a Lagrangian framework for evolving point-cloud data through localized B-spline surface fitting. The method is structured around three key components: local B-spline surface fitting, which constructs smooth B-spline patches over scattered point-cloud data using localized fitting and avoids the need for global parameterization; evolution of the velocity-guided control points, which evolves control points and data points using a velocity field typically defined through geometric quantities such as surface normals and curvature, and refines them via Gauss-Seidel optimization to maintain data interpolation; and adaptive refinement, which dynamically inserts knots to refine surface resolution and adds or removes points to control sampling density and enhance regularity during evolution. 

The remainder of this paper is structured as follows. Section~\ref{sec:preliminaries} reviews the mathematical preliminaries, including the Lagrangian framework, B-spline surface representations, and differential geometric formulations. Section \ref{sec:adaptive_bspline_surface} presents our adaptive B-spline surface fitting method, including local interpolation and refinement of control points. Section \ref{sec:numerical_results} reports numerical experiments that explore key parameters and evaluate the performance of the proposed framework through dynamic surface simulations. Finally, Section \ref{sec:conclusion} summarizes our findings and outlines the directions for future research.

%%%%%%%%%%%%%%%%%%%%%%%%%%%%%%%%%%%%%%%%%%%%%%%%%%%%
\section{Fundamentals of Lagrangian and Spline Geometry}
\label{sec:preliminaries}
This section establishes the mathematical background and notation for our point-cloud evolution framework. We first outline the Lagrangian representation of evolving manifolds, then introduce tensor-product B-spline surfaces, discuss basis-function types, and review classical surface interpolation techniques.

\subsection{Lagrangian Dynamics of Evolving Manifolds}\label{sec:general_framework}
Let $ \mathcal{M}(t) \subset \mathbb{R}^d $ denote a family of embedded, orientable, smooth manifolds of dimension $ d_{\mathcal{M}} < d $, evolving over $ t \in [0, T] $. The evolution of $ \mathcal{M}(t) $ is induced by a time-dependent diffeomorphism
\[
\Phi_t : \mathcal{M}(0) \to \mathcal{M}(t) \subset \mathbb{R}^d,
\]
such that for any material coordinate $ \vec{X} \in \mathcal{M}(0) $, the trajectory $ \vec{x}(t) = \Phi_t(\vec{X}) $ satisfies the flow equation
\begin{equation}
    \frac{d\vec{x}}{dt} = \vec{V}(\vec{x}, t), \qquad \vec{x}(0) = \vec{X},
\end{equation}
where $ \vec{V}: \bigcup_{t \in [0, T]} \mathcal{M}(t) \times \{t\} \to \mathbb{R}^d$ is a smooth velocity field that prescribes the geometric evolution.

When a closed-form description of $ \mathcal{M}(t) $ is unavailable for $t > 0$, we approximate the manifold by a discrete time-dependent sampling:
\begin{equation}
    \mathcal{X}(t) := \{ \vec{x}_i(t) \}_{i=1}^{N} \subset \mathcal{M}(t),
\end{equation}
advected by the flow. A standard explicit Lagrangian time-stepping scheme is:
\begin{equation}
    \vec{x}_i(t + \Delta t) = \vec{x}_i(t) + \Delta t \cdot \vec{V}(\vec{x}_i(t), t).
    \label{eq:moving_datavelocity}
\end{equation}

This discrete representation must accurately approximate local differential invariants, such as normals and principal curvatures, without a global coordinate chart. To support long-term evolution while maintaining geometric structure, the point set $ \mathcal{X}(t) $ is adaptively refined using local geometric criteria, maintaining both intrinsic and extrinsic characteristics of the evolving manifold.

Our previous work in \cite{ammad2025eabe} developed this Lagrangian framework specifically for evolving plane curves represented by point clouds. In that curve version, we partitioned the points into overlapping stencils, fit local B-spline interpolants to approximate geometric quantities like curvature and normals, and evolved both the data points and B-spline coefficients (which act like geometric "control points") under curvature-driven flows. Adaptive knot insertion and point redistribution were used to handle density changes without frequent re-interpolation, making it efficient for 1D manifolds. This allowed seamless simulations of curvature flows and coupled reaction-diffusion systems on curves. However, extending to codimension-one surfaces in $\R^3$ requires handling 2D parameterizations, tensor-product basis for surface patches, and higher-dimensional geometric invariants like mean curvature, while preserving the core idea of evolving control structures alongside the data. The next section details these extensions, building on the foundational concepts outlined here.

\subsection{Tensor-Product B-Spline Surfaces and Classical Interpolation}
\label{sec:bspline_surface}
B-splines provide a compact representation of smooth surfaces, central to our interpolation on evolving point clouds. A tensor-product B-spline surface of bi-degree $(p, q)$ is defined by a grid of control points $\{\vec{P}_{i,j} \in \mathbb{R}^d\}_{i,j=1}^{m,n}$ and nondecreasing knot vectors $\Psi = \{\psi_1, \dots, \psi_{m+p+1}\}$ and $\Theta = \{\theta_1, \dots, \theta_{n+q+1}\}$. The univariate basis functions $\phi_i^p(u)$ and $\varphi_j^q(v)$ are defined recursively via Cox–de Boor relations: for degree 0, $\phi_i^0(u) = 1$ if $\psi_i \le u < \psi_{i+1}$ and 0 otherwise; for higher degrees, $\phi_i^p(u) = \frac{u - \psi_i}{\psi_{i+p} - \psi_i} \phi_i^{p-1}(u) + \frac{\psi_{i+p+1} - u}{\psi_{i+p+1} - \psi_{i+1}} \phi_{i+1}^{p-1}(u)$, with terms having zero denominators taken as zero (an analogous recurrence applies to $\varphi_j^q(v)$ over $\Theta$). This yields the surface $ \vec{\mathcal{S}} : [0,1]^2 \to \mathbb{R}^d$ by:

\begin{equation}
    \vec{\mathcal{S}}(u,v) = \sum_{i=1}^{m} \sum_{j=1}^{n} \vec{P}_{i,j} \, \phi_i^p(u) \, \varphi_j^q(v).
    \label{eq:bspline_surface}
\end{equation}
The surface inherits local support and smoothness from the basis, with continuity $C^{p-1}$ in $u$ and $C^{q-1}$ in $v$ at minimal knot multiplicities \cite{ammad2019cubic,piegl2012nurbs}.

Basis types, determined by knot vectors, affect continuity and interpolation. Periodic basis ensure cyclic continuity for closed surfaces but do not interpolate boundaries; they satisfy $\psi_{m+j} = \psi_{m+j-1} + (\psi_j - \psi_{j-1})$ for $j = 1, \dots, p+1$ \cite{marsh2005applied}. Open uniform basis enforce boundary interpolation by repeating the first and last $p+1$ knots, making them suitable for precise boundary control and interpolating corner points like $\vec{\mathcal{S}}(0,0) = \vec{P}_{1,1}$. In this work, we use open uniform basis for their boundary properties, though periodic ones suit closed topologies. Complementary approaches such as generalized trigonometric Bézier representations offer alternative formulations with enhanced shape control for specific modeling scenarios \cite{ammad2022novel}.

For classical interpolation on a given $m\times n$ data grid ${\vec{Q}_{i,j}}$ with prescribed parameter samples ${\tilde{u}_j}$ and ${\tilde{v}_i}$ (e.g., chord-length or centripetal along isoparametric lines):
\begin{equation}
\tilde{u}_1 = 0, \quad \tilde{u}_j = \tilde{u}_{j-1} + \frac{\|\vec{Q}_{j} - \vec{Q}_{j-1}\|^\alpha}{\sum_{\ell=2}^{n} \|\vec{Q}_{\ell} - \vec{Q}_{\ell-1}\|^\alpha}, \quad j = 2, \dots, n.
    \label{eq:classical+para}
\end{equation}
Knot vectors are constructed by averaging, yielding a sparse linear system for $\vec{P}_{i,j}$, solvable via iterative methods like Conjugate Gradient \cite{botsch2005+CG}. Surface quality depends on parameterization; refinements adjust knots or parameters to minimize residuals in high-curvature regions \cite{michel2021new,guo2026two}.

%%%%%%%%%%%%%%%%%%%%%%%%%%%%%%%%%%%%%%%%%%%%%%%%%%%%%%%%
%%%%%%%%%%%%%%%%%%%%%%%%%%%%%%%%%%%%%%%%%%%%%%%%%%%%%%%%
\subsection{Differential Geometry of B-spline Surfaces}
\label{sec:nandkappa}
The differential geometry of B-spline surfaces governs their intrinsic and extrinsic behavior, and underpins curvature-driven modeling and geometric evolution. Let $  \vec{\mathcal{S}}(u,v) : \Omega \subset \mathbb{R}^2 \rightarrow \mathbb{R}^d $ be a $ C^2 $ parametric surface defined by~\eqref{eq:bspline_surface}.

The tangent vectors at any point $ (u,v) \in \Omega $ are given by the first-order partial derivatives: $ \vec{\mathcal{S}}_u = \partial  \vec{\mathcal{S}}/\partial u,  \vec{\mathcal{S}}_v = \partial  \vec{\mathcal{S}}/\partial v.$
They span the  tangent plane and define the first fundamental form:
\[
I = E \, du^2 + 2F \, du \, dv + G \, dv^2,
\]
where $E = \langle  \vec{\mathcal{S}}_u,  \vec{\mathcal{S}}_u \rangle$, $F = \langle  \vec{\mathcal{S}}_u,  \vec{\mathcal{S}}_v \rangle$, and $G = \langle  \vec{\mathcal{S}}_v,  \vec{\mathcal{S}}_v \rangle$.
This metric describes intrinsic geometric quantities such as arc length, angle, and area. The unit normal vector $ \vec{n}(u,v) $ is defined (in $ \mathbb{R}^3 $) by:
\begin{equation}
\vec{n}(u,v) = \frac{ \vec{\mathcal{S}}_u \times  \vec{\mathcal{S}}_v}{\| \vec{\mathcal{S}}_u \times  \vec{\mathcal{S}}_v\|},
\label{eq:normalvector}
\end{equation}
providing the orientation necessary for curvature analysis. In higher dimensions, this can be generalized using exterior algebra and orthonormalization procedures.

To characterize the surface’s extrinsic curvature, we compute the second-order derivatives: $ \vec{\mathcal{S}}_{uu},  \vec{\mathcal{S}}_{uv}, \vec{\mathcal{S}}_{vv},$ and project them onto the normal to obtain the second fundamental form:
\[
II = L \, du^2 + 2M \, du \, dv + N \, dv^2,
\]
where $L = \langle  \vec{\mathcal{S}}_{uu}, \vec{n} \rangle$, $M = \langle  \vec{\mathcal{S}}_{uv}, \vec{n} \rangle$, and $N = \langle  \vec{\mathcal{S}}_{vv}, \vec{n} \rangle$.
The mean curvature $ H $, representing the average principal curvature at a point, is given by the classical expression:
\begin{equation}
H = \frac{EN - 2FM + GL}{2(EG - F^2)}.
\label{eq:mean_curvature}
\end{equation}

These quantities are efficiently evaluated from B-spline representations using recursive differentiation of the basis functions. They serve as key inputs to variational models and geometric evolution algorithms.

%%%%%%%%%%%%%%%%%%%%%%%%%%%%%%%%%%%%%%%%%%%%%%%%%%%%%%%%%%%%%%%%%%%%%%%%%%%
\section{Localized B-spline Surface Evolution for Dynamic point-clouds}
\label{sec:adaptive_bspline_surface}
To extend the curve-evolution algorithm from \cite{ammad2025eabe} to surfaces, we first outline its core steps in pseudocode form (Algorithm~\ref{alg:curve_version}). This provides a clear reference for what can be reused directly and what requires adaptation for higher-dimensional manifolds.

\begin{algorithm}[H]
\caption{Curve Evolution Algorithm from \cite{ammad2025eabe}}
\label{alg:curve_version}
\begin{algorithmic}[1]
\State \textbf{Input:} Initial point cloud $\mathcal{X}(0)$, time step $\Delta t$, final time $T$, tolerances (e.g., deviation $\epsilon_{\mathrm{tol}}$, interpolation error $\tau$)
\State Partition $\mathcal{X}(0)$ into non-overlapping core covers and extend to overlapping stencils
\For{each stencil}
    \State Parameterize points using chord-length method
    \State Fit local B-spline interpolant (open uniform basis)
    \State Refine control points via knot insertion if deviation $>\epsilon_{\mathrm{tol}}$
    \State Compute normals and curvature from B-spline derivatives
\EndFor
\While{$t < T$}
    \For{each point $\vec{x}_i(t) \in \mathcal{X}(t)$}
        \State Update $\vec{x}_i(t + \Delta t) = \vec{x}_i(t) + \Delta t \cdot \vec{V}(\vec{x}_i(t), t)$ (e.g., curvature-driven)
        \State Evolve associated control points similarly
    \EndFor
    \For{each stencil}
        \State Evaluate interpolation error; if $>\tau$, optimize control points via Gauss-Seidel
        \State Insert/remove/redistribute points to maintain density
    \EndFor
    \State $t \gets t + \Delta t$
\EndWhile
\State \textbf{Output:} Evolved point cloud $\mathcal{X}(T)$
\end{algorithmic}
\end{algorithm}

Several components of Algorithm~\ref{alg:curve_version} can be used almost exactly for surfaces: the Lagrangian point updates (line 11) and control point evolution (line 12) translate directly, as they operate on discrete points and coefficients without depending on dimensionality. Other parts are reused but with extensions, such as the Gauss–Seidel refinement on control points (line 15; extended to handle surface control nets) and adaptive point management (line 16; extended to update points by tensor-product knot insertion in 2D). However, further extensions are needed for surface-specific aspects. Partitioning now creates 2D patches instead of 1D stencils (Section~\ref{sec:patches}), parameterization shifts to projection-based methods in a 2D domain (Section~\ref{sec:bspline_surface_interpolation}), B-spline fitting uses tensor-products rather than univariate basis (also Section~\ref{sec:bspline_surface_interpolation}), and geometric computations involve surface differentials like mean curvature in $\R^3$ (Section~\ref{sec:nandkappa}). We also add conditioning-aware rotations for regularity (Section~\ref{sec:stability_bsplines}) and Greville-based refinement metrics tailored to surfaces (Section~\ref{subsec:control_point_refinement}). These changes build on the curve method while handling the added complexity of 2D manifolds.
%%%%%%%%%%%%%%%%%%%%%%%%%%%%%%%%%%%%%%%%%%%%%%%%%%%%%%%%%%%
\subsection{Constructing Overlapping Patches for Local Surface Approximation}
\label{sec:patches}

Let $\mathcal{X}(0) = \{\vec{x}_i(0)\}_{i=1}^{N} \subset \mathbb{R}^d$ denote a discrete point-cloud sampled from a smooth initial surface $ \vec{\mathcal{S}}(0) \subset \mathbb{R}^d$. To enable localized surface fitting and geometric estimation, we partition $\mathcal{X}(0)$ into a collection of overlapping patches $\{\mathcal{P}_k\}_{k=1}^{M}$, each serving as a computational domain for B-spline interpolation and curvature-driven evolution. The construction of these patches proceeds in two stages: first, a non-overlapping decomposition into disjoint core patches is obtained; then, each core is extended to form an overlapping neighborhood that facilitates smooth geometric transitions.

In the first stage, disjoint \emph{core patches} $\{\mathcal{P}_k^{\mathrm{core}}\}_{k=1}^{M}$ are constructed using an iterative algorithm of k nearest neighbors (kNN). At each iteration, a seed point $\vec{x}_{s,k}$ is selected from the set of unassigned points, and its $m_{c,k}$ nearest neighbors define the core:
\begin{equation}
\mathcal{P}_k^{\mathrm{core}} = \left\{ \vec{x}_j \in \mathcal{X}(0) \mid \vec{x}_j \in \text{kNN}(\vec{x}_{s,k}, m_{c,k}) \right\}.
\label{eq:core_patch}
\end{equation}
Once a core patch is identified, its constituent points are removed from the candidate set to ensure disjointness. This process continues until each point in $\mathcal{X}(0)$ is assigned to exactly one core patch, thereby satisfying
\begin{equation*} \label{eq:cover}
\bigcup_{k=1}^{M} \mathcal{P}_k^{\mathrm{core}} = \mathcal{X}(0), \qquad \mathcal{P}_i^{\mathrm{core}} \cap \mathcal{P}_j^{\mathrm{core}} = \emptyset, \quad \text{for } i \neq j.
\end{equation*}

To ensure geometric representativeness, each core patch is associated with a well-defined center. We first compute the centroid
\[
\tilde{\vec{x}}_{c,k} = \frac{1}{m_{c,k}} \sum_{\vec{x}_j \in \mathcal{P}_k^{\mathrm{core}}} \vec{x}_j,
\]
and define the actual center $\vec{x}_{c,k}$ as the point in $\mathcal{P}_k^{\mathrm{core}}$ nearest to $\tilde{\vec{x}}_{c,k}$:
\[
\vec{x}_{c,k} = \arg\min_{\vec{x} \in \mathcal{P}_k^{\mathrm{core}}} \|\vec{x} - \tilde{\vec{x}}_{c,k}\|.
\]
This ensures that the center lies within the original data and can serve as a stable reference for further neighborhood construction.
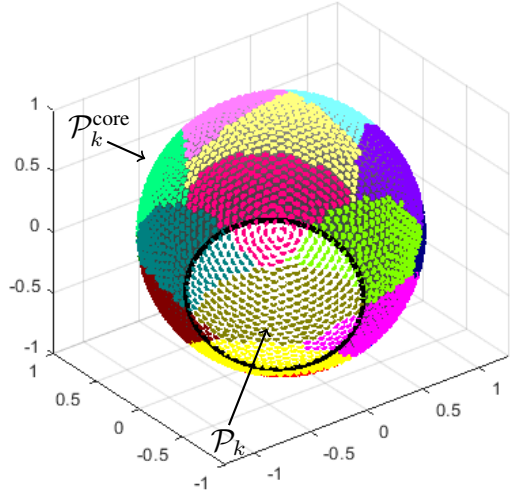
\begin{figure}
\centering
\begin{tikzpicture}
    \node[anchor=south west, inner sep=0] (img) at (0,0) {\includegraphics[width=10cm]{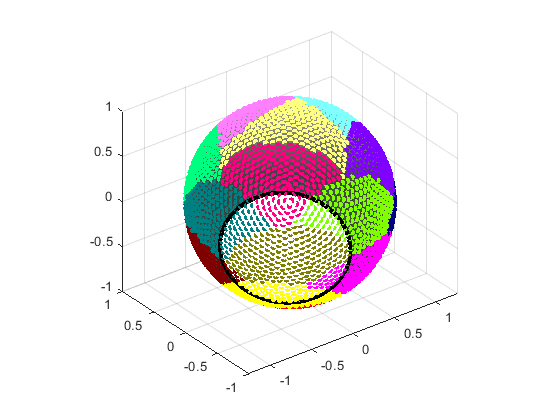}};
    \begin{scope}[x={(img.south east)},y={(img.north west)}]
        \node at (0.28,0.7) {$\mathcal{P}_k^\text{core}$};
        \draw[->, thick] (0.29,0.68) -- (0.34,0.65);
        \node at (0.45,0.15) {$\mathcal{P}_k$};
        \draw[->, thick] (0.45,0.18) -- (0.50,0.35);
    \end{scope}
\end{tikzpicture}
\caption{
Visualization of the patch construction on a point cloud sampled from a spherical surface. Each distinct color denotes a disjoint core patch $ \mathcal{P}_k^{\mathrm{core}} $, with no point shared across patches. The highlighted circular region illustrates an extended patch $ \mathcal{P}_k $, which includes its core (marked by uniform color) together with additional neighboring points from adjacent patches.
}
\label{fig:patching}
\end{figure}

In the second stage, each core patch is expanded to an \emph{overlapping patch} $\mathcal{P}_k$ by augmenting it with $m_{b,k}$ additional points from neighboring regions. A secondary kNN search centered on $\vec{x}_{c,k}$ yields the boundary region:
\begin{equation}
\mathcal{P}_k = \mathcal{P}_k^{\mathrm{core}} \cup \mathcal{P}_k^{\mathrm{bdy}}, \qquad
\mathcal{P}_k^{\mathrm{bdy}} = \left\{ \vec{x}_j \in \mathcal{X}(0) \setminus \mathcal{P}_k^{\mathrm{core}} \mid \vec{x}_j \in \text{kNN}(\vec{x}_{c,k}, m_{b,k}) \right\}.
    \label{eq:core+overlap}
\end{equation}
Each resulting patch $\mathcal{P}_k$ contains $m_k = m_{c,k} + m_{b,k}$ points and overlaps with adjacent patches, allowing for smooth transitions and consistent local geometry.

This construction is illustrated in Figure~\ref{fig:patching}. At any time $t$, the dynamic point cloud $\mathcal{X}(t)$ remains covered by the collection of core patches, with
\[
\mathcal{X}(t) = \bigcup_k \mathcal{P}_k^{\mathrm{core}} \subseteq \bigcup_k \mathcal{P}_k,
\]
guaranteeing that each point contributes to at least one local surface reconstruction.

This localized partitioning provides the structural foundation for our surface evolution framework. In the following subsection, we construct B-spline surface interpolants over each patch and describe how these functions are used to extract differential geometric quantities that drive the motion of the underlying point cloud.
%%%%%%%%%%%%%%%%%%%%%%%%%%%%%%%%%%%%%%%%%%%%%%%%%%%%%%%%%%%%%%%
\subsection{B-spline Interpolation Technique for Dynamic Surface Modeling}
\label{sec:bspline_surface_interpolation}

Given a fixed patch $ \mathcal{P}_k \subset \mathcal{X}(t) $, we now describe how to construct a smooth B-spline surface interpolant over its local neighborhood. Since the points in $ \mathcal{P}_k $ are scattered and lack a regular structure, we begin by embedding them into a two-dimensional parametric domain suitable for tensor-product B-spline interpolation.

For structured data arranged in grids or rows, such as height fields or structured meshes, classical chordal and centripetal methods are commonly employed \cite{piegl2012nurbs}, as described in Equation~\eqref{eq:classical+para}. For unstructured geometric data, including polygonal meshes and point clouds, a range of techniques has been developed based on harmonic maps, conformal flattening, and local tangent space projections \cite{liu2008local+mesh, levy2023+mesh, zhang2010+pcloud, meng2013+pcloud, choi2016+pcloud}.  We adopt a local projection-based approach that aligns each patch to its principal directions and normalizes the resulting planar coordinates into a standard parametric domain.

Let $ \mathcal{P}_k = \{\vec{x}_r\}_{r=1}^N \subset \mathbb{R}^3 $ denote the set of points in the $ k $-th patch. To construct a parametric domain, we first calculate the empirical mean
\[
\mathbb{E}[\vec{x}] = \frac{1}{N} \sum_{r=1}^{N} \vec{x}_r, \quad \text{and the covariance matrix} \quad \mathcal{C} = \frac{1}{N} \sum_{r=1}^{N} (\vec{x}_r - \mathbb{E}[\vec{x}]) (\vec{x}_r - \mathbb{E}[\vec{x}])^\top.
\]
and let $ \{\mathbf{e}_1, \mathbf{e}_2, \mathbf{e}_3\} \subset \mathbb{R}^3 $ be the eigenvectors of $ \mathcal{C} $, ordered by decreasing eigenvalue. These define a local orthonormal frame; $\mathbf{e}_3$ approximates the local surface normal.

Define rotation matrix $ \mathbf{R} = [\mathbf{e}_1 \ \mathbf{e}_2 \ \mathbf{e}_3] \in \mathrm{SO}(3) $, and rotate the data into this canonical frame:
\[
\vec{x}_r' = \mathbf{R}^{\top} (\vec{x}_r - \mathbb{E}[\vec{x}]) \in \mathbb{R}^3.
\]
This step is illustrated in Figure~\ref{fig:main}(b), where the rotated patch is aligned with the coordinate axes. Next, we project each transformed point onto the $ (x,y) $-plane using the projection matrix
\[
\mathbf{\Pi} = \begin{bmatrix} 1 & 0 & 0 \\ 0 & 1 & 0 \end{bmatrix} \in \mathbb{R}^{2 \times 3}, \quad \text{yielding} \quad \mathbf{q}_r = \mathcal{T}(\vec{x}_r) = \mathbf{\Pi} \vec{x}_r' \in \mathbb{R}^2.
\]
as shown in Figure~\ref{fig:main}(c). This yields a planar embedding of the patch suitable for parameterization. To normalize the coordinates into the unit square $ [0,1]^2 $, we compute the axis-aligned bounding box:
\[
x_{\min}' = \min_r [\mathbf{q}_r]_1, \quad x_{\max}' = \max_r [\mathbf{q}_r]_1, \quad 
y_{\min}' = \min_r [\mathbf{q}_r]_2, \quad y_{\max}' = \max_r [\mathbf{q}_r]_2,
\]
and define the scaled parameter values
\[
u_r = \frac{[\mathbf{q}_r]_1 - x_{\min}'}{x_{\max}' - x_{\min}'}, \qquad
v_r = \frac{[\mathbf{q}_r]_2 - y_{\min}'}{y_{\max}' - y_{\min}'}.
\]

Thus, each data point $ \vec{x}_r \in \mathbb{R}^3 $ is associated with a parametric coordinate $ (u_r, v_r) \in [0,1]^2 $, as shown in Figure~\ref{fig:main}(c).
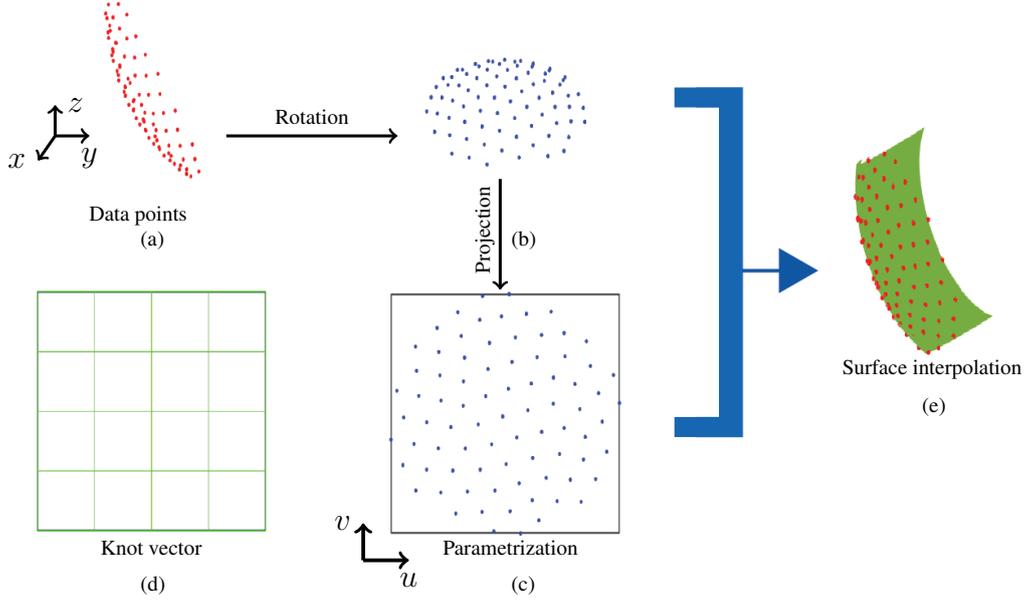
\begin{figure}[t]
\centering
\begin{tikzpicture}
    \node[anchor=south west, inner sep=0] (img) at (0,0) {\includegraphics[width=15cm]{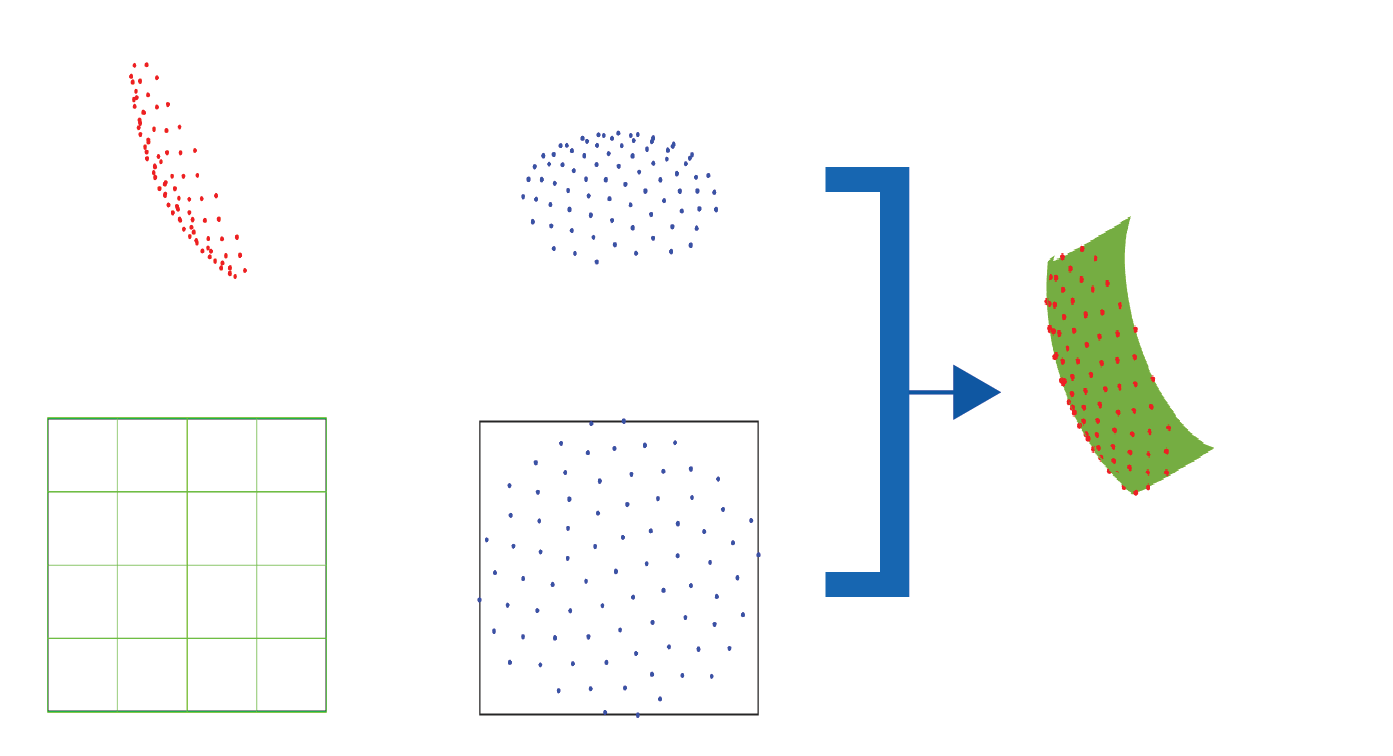}};
    \begin{scope}[x={(img.south east)}, y={(img.north west)}]
        \draw[line width=0.4mm, ->, black] (0.05, 0.7) -- (0.08, 0.7) node[anchor=north, black] {$y$}; 
        \draw[line width=0.4mm, ->, black] (0.05, 0.7) -- (0.05, 0.75) node[anchor=west, black] {$z$}; 
        \draw[line width=0.4mm, ->, black] (0.05, 0.7) -- (0.035, 0.66) node[anchor=east, black] {$x$}; 
        \draw[line width=0.5mm, ->, black] (0.32, 0.01) -- (0.36, 0.01) node[anchor=north, black] {$u$}; 
        \draw[line width=0.5mm, ->, black] (0.32, 0.01) -- (0.32, 0.07) node[anchor=east, black] {$v$}; 
        \node[anchor=west, black] at (0.07, 0.57) {\scriptsize Data points};
        \node[anchor=west, black] at (0.38, 0.03) {\scriptsize Parametrization};
        \node[anchor=west, black] at (0.08, 0.03) {\scriptsize Knot vector};
        \node[anchor=west, black] at (0.73, 0.32) {\scriptsize Surface interpolation };
        \node[anchor=west, black] at (0.115, 0.53) {\scriptsize (a)};
        \node[anchor=west, black] at (0.44, 0.53) {\scriptsize (b)};
        \node[anchor=west, black] at (0.44, -0.03) {\scriptsize (c)};
        \node[anchor=west, black] at (0.115, -0.03) {\scriptsize (d)};
        \node[anchor=west, black] at (0.8, 0.26) {\scriptsize (e)};
        \draw[line width=0.4mm, ->, black] (0.2, 0.7) -- (0.35, 0.7) node[midway, above, black, font=\scriptsize] {Rotation};
        \draw[line width=0.4mm, ->, black] (0.44, 0.63) -- (0.44, 0.45) node[midway, left, black, rotate=90, xshift=0.8cm, yshift=0.2cm] {\scriptsize Projection};
    \end{scope}
\end{tikzpicture}
\caption{B-spline surface interpolation pipeline. (a) Input data points in patch $ \mathcal{P}_k \subset \mathbb{R}^3 $ (b) rotation to local frame; (c) projection to a parametric plane (d) uniform knot vector construction (e) resulting B-spline interpolant.}
\label{fig:main}
\end{figure}
We now define a surface interpolant over the normalized domain. Let $ \{ \phi_\ell(u,v) \}_{\ell=1}^{s} $ be scalar-valued basis functions on $ [0,1]^2 $, and let $ \vec{P}_\ell \in \mathbb{R}^3 $ be the associated control vectors. The interpolating surface is 
\[
 \vec{\mathcal{S}}(u,v) = \sum_{\ell=1}^{N} \vec{P}_\ell \, \phi_\ell(u,v).
\]
To satisfy the interpolation constraint at each parameter location $ (u_r, v_r) $, we impose
\begin{equation}
S_j(u_r, v_r) = \sum_{\ell=1}^{N} [\vec{P}_\ell]_j \, \phi_\ell(u_r, v_r) = [\vec{x}_r]_j, \quad \text{for } j = 1,2,3, \quad r = 1,\dots,N.
\label{eq:b-sp_surafce_inter}
\end{equation}
Solving this system for each coordinate $ j $ determines the full set of control vectors
\[
\vec{P}_\ell = \begin{bmatrix} [\vec{P}_\ell]_1 \\ [\vec{P}_\ell]_2 \\ [\vec{P}_\ell]_3 \end{bmatrix} \in \mathbb{R}^3.
\]

When the basis functions $ \phi_\ell(u,v) $ are chosen as tensor-product B-spline functions over a uniform knot vector, we obtain a smooth B-spline surface interpolant. A representative knot layout is shown in Figure~\ref{fig:main}(d), and the resulting interpolated surface is shown in Figure~\ref{fig:main}(e).

Once the tensor-product B-spline surface S(u,v) is constructed for each patch, we obtain a compact, differentiable representation suitable for geometric processing. However, the conditioning of the interpolation system can degrade under anisotropic sampling or poorly distributed parameter values, leading to instability. To mitigate this, the next section introduces a conditioning-aware reformulation that preserves the B-spline basis and knots while improving the interpolation matrix via parameter-space rotation and normalization.
%%%%%%%%%%%%%%%%%%%%%%%%%%%%%%%%%%%%%%%%%%%%%%%%%%%%%%%%%
\subsection{Conditioning-Aware Reformulation for B-Spline Surface Interpolation}
\label{sec:stability_bsplines}
The linear system arising from interpolation may be ill-conditioned when the parameterization exhibits poor aspect ratios or irregular sampling. We present a reformulation that preserves the B-spline basis while improving the conditioning of the interpolation matrix, without modifying the knot vectors (cf. \cite{wang2008+knotgeneration}).

Let \( \mathcal{P}_k = \{\vec{x}_r\}_{r=1}^{N} \subset \mathbb{R}^3 \) be the set of points in a local patch, and let \( \{(u_r, v_r)\}_{r=1}^{N} \subset [0,1]^2 \) denote their associated parametric coordinates. The tensor-product B-spline surface is given by
\begin{equation}
\vec{x}_r =  \vec{\mathcal{S}}(u_r, v_r) = \sum_{i=1}^{m} \sum_{j=1}^{n} \phi_i^p(u_r) \,\varphi_j^q(v_r)\, \vec{P}_{i,j}, \quad r = 1, \dots, N,
\end{equation}
where \( \phi_i^p \) and \( \varphi_j^q \) are univariate B-spline basis functions of degrees \(p\) and \(q\), and \( \vec{P}_{i,j} \in \mathbb{R}^3 \) are control points.

To highlight the tensor-product structure, we write the evaluation hierarchically:
\[
\vec{x}_r = \sum_{i=1}^{m} \phi_i^p(u_r)\,\vec{C}_i(v_r),
\qquad
\vec{C}_i(v_r) = \sum_{j=1}^{n} \varphi_j^q(v_r)\,\vec{P}_{i,j},
\]
so that the surface can be viewed as a family of B-spline curves in the \(v\)-direction, modulated by a B-spline basis in the \(u\)-direction.

Collecting all data points into
\[
X = 
\begin{bmatrix}
\vec{x}_1^\top \\ \vdots \\ \vec{x}_N^\top
\end{bmatrix} \in \mathbb{R}^{N \times 3},
\qquad
P = 
\begin{bmatrix}
\vec{P}_{1,1}^\top & \cdots & \vec{P}_{m,n}^\top
\end{bmatrix}^\top \in \mathbb{R}^{mn \times 3},
\]
and defining the evaluation matrix \( M \in \mathbb{R}^{N \times mn} \) by
\[
M_{r,(i,j)} = \phi_i^p(u_r)\,\varphi_j^q(v_r),
\qquad 1 \le r \le N,\ 1 \le i \le m,\ 1 \le j \le n,
\]
we obtain the linear system
\begin{equation}
X = M P.
\label{eq:interp_system}
\end{equation}

In many of our examples we choose \(N = mn\), so that \(M \in \mathbb{R}^{N\times N}\) is square and the surface interpolates the data on each patch. A necessary condition for a unique interpolant in this case is that \(M\) has full row rank,
\[
\operatorname{rank}(M) = N.
\]
For general scattered configurations, however, there are few theoretical results that guarantee full rank of the associated tensor-product B-spline matrix beyond the regular-grid case. Even if all parameter pairs \(\{(u_r,v_r)\}\) are distinct, the matrix \(M\) may still be rank-deficient or severely ill-conditioned when the sampling is highly irregular.

Our aim here is not to derive abstract rank conditions, but rather to improve the numerical conditioning of the interpolation problem defined by \eqref{eq:interp_system}. To this end, before forming $M$, we apply a rotation-plus-normalization of the parameter pairs $\{(u_r,v_r)\}$ and select, among a finite set of candidate rotations, the one that minimizes the spectral condition number $\kappa_2(M)$. This transformation leaves the spline space itself unchanged: the polynomial degrees and knot vectors remain fixed, the number and arrangement of control points are not altered, and no additional constraints are introduced. What changes is only the parametrization of the given data within the same unit square $[0,1]^2$; the coordinates $(u_r,v_r)$ are mapped to new coordinates that lead to a better-conditioned evaluation of the same tensor-product basis. In other words, the ansatz $\vec{\mathcal{S}}$ and the unknowns $P$ are the same, but the coordinate frame in which the basis functions are sampled is chosen in a conditioning-aware manner.

This conditioning-aware preprocessing should be regarded as a practical heuristic to stabilize the interpolation system. It can be combined with more robust fitting strategies when necessary. In particular, if the smallest singular value of \(M\) is below a prescribed tolerance, instead of attempting to invert \(M\) we may solve a least-squares problem with more data points than unknowns,
\[
P = \arg\min_{\tilde P \in \mathbb{R}^{mn\times 3}} \|M \tilde P - X\|_2^2,
\]
and interpret the singular values of \(M\) as an indicator of the quality of the approximation. Such a least-squares formulation is well known to be more reliable for highly irregular or noisy point sets and provides a natural complement to our conditioning-aware reformulation.
%%%%%%%%%%%%%%%%%%%%%%%%%%
\begin{condition}[Rotation--normalization conditioning]
\label{cond:stability}
Let $\{(u_r,v_r)\}_{r=1}^N \subset [0,1]^2$ be parameter values for a patch, and let the spline basis (degree and knot vectors) be fixed a priori and independent of $\omega$. For $\omega\in[0,2\pi]$, define the rotation
\[
T(\omega)=
\begin{bmatrix}
\cos\omega & -\sin\omega\\
\sin\omega & \cos\omega
\end{bmatrix}\in\mathrm{SO}(2),
\qquad
\begin{bmatrix}\hat u_r(\omega)\\ \hat v_r(\omega)\end{bmatrix}
= T(\omega)\begin{bmatrix}u_r\\ v_r\end{bmatrix}.
\]
Define normalized parameters by independent affine min--max scalings along each axis:
\[
u_r(\omega)=
\frac{\hat u_r(\omega)-\min_{1\le s\le N} \hat u_s(\omega)}
{\max_{1\le s\le N} \hat u_s(\omega)-\min_{1\le s\le N} \hat u_s(\omega)},
\qquad
v_r(\omega)=
\frac{\hat v_r(\omega)-\min_{1\le s\le N} \hat v_s(\omega)}
{\max_{1\le s\le N} \hat v_s(\omega)-\min_{1\le s\le N} \hat v_s(\omega)},
\]
whenever the denominators are positive. Let $M(\omega)\in\mathbb{R}^{N\times N}$ denote the interpolation matrix obtained by evaluating the chosen tensor-product B-spline basis at $\{(u_r(\omega),v_r(\omega))\}_{r=1}^N$, and let $\kappa_2(\cdot)$ denote the $2$-norm condition number. We assume that there exist $\omega^*\in[0,2\pi]$ and $\tau>0$ such that
\[
\kappa_2\big(M(\omega^*)\big)\le \tau.
\]
This is a numerical assumption: for a given patch, it can be checked in practice by sampling $\omega$ on a grid and computing the corresponding condition numbers.
\end{condition}

\begin{proposition}[Square system solved at the best-conditioned rotation]
\label{prop:best_rotation}
Let $M(\omega)\in\mathbb{R}^{N\times N}$ be the square interpolation matrix constructed from the normalized parameter pairs $\{(u_r(\omega),v_r(\omega))\}$ obtained by rotation $T(\omega)\in\mathrm{SO}(2)$ followed by axis-wise min--max scaling, with $\omega\in[0,2\pi]$. Define
\[
\omega^* \in \arg\min_{\omega\in[0,2\pi]} \kappa_2\big(M(\omega)\big),
\]
and assume that $M(\omega^*)$ is nonsingular. Then the interpolation system
\[
X \;=\; M(\omega^*)\,P
\]
admits the unique solution $P=M(\omega^*)^{-1}X$. Moreover, for any data perturbation $\Delta X$,
\[
\|\Delta P\|_2 \;\le\; \|M(\omega^*)^{-1}\|_2\,\|\Delta X\|_2 \;=\; \frac{1}{\sigma_{\min}(M(\omega^*))}\,\|\Delta X\|_2,
\]
and, if $X\neq 0$,
\[
\frac{\|\Delta P\|_2}{\|P\|_2}
\;\le\;
\kappa_2\!\big(M(\omega^*)\big)\,\frac{\|\Delta X\|_2}{\|X\|_2}.
\]
Within this rotation-plus-normalization family, the choice of $\omega^*$ therefore minimizes the usual condition-number-based stability bound.
\end{proposition}

\begin{proof}
Let $M:=M(\omega^*)$. Since $M$ is invertible, $P=M^{-1}X$ is the unique solution of $X=MP$. For a perturbation $\Delta X$, one has $\Delta P = M^{-1}\Delta X$, hence
\[
\|\Delta P\|_2 \le \|M^{-1}\|_2\,\|\Delta X\|_2 = \frac{1}{\sigma_{\min}(M)}\,\|\Delta X\|_2.
\]
If $X\neq 0$, then $P=M^{-1}X$ and
\[
\frac{\|\Delta P\|_2}{\|P\|_2}
\le
\frac{\|M^{-1}\|_2\,\|\Delta X\|_2}{\|M^{-1}X\|_2}
\le
\|M^{-1}\|_2\,\|M\|_2\,\frac{\|\Delta X\|_2}{\|X\|_2}
=
\kappa_2(M)\,\frac{\|\Delta X\|_2}{\|X\|_2}.
\]
Because $\omega^*$ minimizes $\kappa_2(M(\omega))$ over $[0,2\pi]$, this bound is the smallest attainable by the proposed rotation-plus-normalization strategy.
\end{proof}

This reformulation introduces no additional degrees of freedom, requires no mesh refinement, and applies independently to each patch. It improves the conditioning of the interpolation system whenever an invertible, well-conditioned matrix $M(\omega^*)$ can be found by the above procedure. In practice we monitor the singular values of $M(\omega)$; if the smallest singular value of $M(\omega^*)$ falls below a prescribed tolerance, we do not attempt to invert $M(\omega^*)$ but instead solve a least-squares problem with more data points than unknowns. The resulting B-spline surface is smooth, but the associated control points are not guaranteed to lie geometrically close to the surface. This proximity is critical because geometric quantities, such as surface normals and curvature, are evaluated on the surface itself and used to update the control points. If control points drift too far, these updates become inconsistent or inaccurate. This motivates the next step in our pipeline, where we enforce the proximity of control points to the surface to ensure reliable geometry-driven evolution.

%%%%%%%%%%%%%%%%%%%%%%%%%%%%%%%%%%%%%%%%%%%%%%%%%%%%%%%%%%%
\subsection{Control Point Refinement Through Distance Minimization}
\label{subsec:control_point_refinement}

Following the conditioning-aware parameterization in Section~\ref{sec:stability_bsplines}, the B-spline surface $  \vec{\mathcal{S}} $ provides a smooth representation of each patch. However, the control points $ \{\vec{P}_{i,j}\}_{1 \leq i \leq m,\ 1 \leq j \leq n} \subset \mathbb{R}^3 $ may not lie near the surface they define. Large deviations, especially in regions of high curvature or anisotropic sampling, can compromise differential-geometric computations.

We quantify how well the control net reflects the surface geometry and use this to guide adaptive refinement via knot insertion.

\begin{definition}[Control Point Deviation]
\label{def:control_point_deviation}
Let $ \vec{P}_{i,j} \in \mathbb{R}^3 $ be a control point. Its deviation from the surface $  \vec{\mathcal{S}}(u,v) $ is defined as
\[
\delta(\vec{P}_{i,j}) := \min_{(u,v) \in [0,1]^2} \left\|  \vec{\mathcal{S}}(u,v) - \vec{P}_{i,j} \right\|,
\]
and the global deviation metric is
\[
\epsilon(\{\vec{P}_{i,j}\},  \vec{\mathcal{S}}) := \max_{1 \leq i \leq m,\ 1 \leq j \leq n} \delta(\vec{P}_{i,j}).
\]
\end{definition}
computing $ \delta(\vec{P}_{i,j}) $ requires solving a non-linear optimization problem for each control point, which is computationally expensive. In practice, we therefore adopt an efficient approximation based on evaluating the surface at the \emph{Greville abscissae}, which is commonly used to associate control points with representative parameter values.

\begin{definition}[Greville-Based Deviation Approximation]
\label{def:greville_deviation}
Let $ \Psi = \{\psi_k\} $ and $ \Theta = \{\theta_k\} $ be the knot vectors in the $ u $- and $ v $-directions, respectively. For each control point $ \vec{P}_{i,j} $, define the associated Greville abscissas as
\[
u_i^{\mathrm{G}} := \frac{1}{p} \sum_{k=1}^{p} \psi_{i+k}, \quad
v_j^{\mathrm{G}} := \frac{1}{q} \sum_{k=1}^{q} \theta_{j+k}.
\]
The Greville-based deviation is then defined by:
\[
\tilde{\delta}(\vec{P}_{i,j}) := \left\|  \vec{\mathcal{S}}(u_i^{\mathrm{G}}, v_j^{\mathrm{G}}) - \vec{P}_{i,j} \right\|, \quad
\tilde{\epsilon}(\{\vec{P}_{i,j}\},  \vec{\mathcal{S}}) := \max_{1 \leq i \leq m,\ 1 \leq j \leq n} \tilde{\delta}(\vec{P}_{i,j}).
\]
\end{definition}

This approximation provides a tractable surrogate for the exact deviation metric. Due to the local support of B-spline basis functions, the Greville abscissae often lie near the region of maximal influence of each control point, making this approximation both effective and computationally efficient.

If $ \tilde{\epsilon}(\{\vec{P}_{i,j}\},  \vec{\mathcal{S}}) > \epsilon_{\mathrm{tol}} $, where $ \epsilon_{\mathrm{tol}} > 0 $ is a prescribed tolerance, we refine the control net by knot insertion. The following lemma ensures that the refined spline space contains the original one.

\begin{lemma}[Knot Insertion for Tensor-Product Surfaces {\cite{lyche2008spline}}]
\label{lemma:knotinsertion_surface}
Let $ \mathbb{S}_{p,q,\Psi,\Theta} $ denote the tensor product spline space defined over the knot vectors $ \Psi $ and $ \Theta $. If $ \tilde{\Psi} \supseteq \Psi $ and $ \tilde{\Theta} \supseteq \Theta $, then
\[
\mathbb{S}_{p,q,\Psi,\Theta} \subseteq \mathbb{S}_{p,q,\tilde{\Psi},\tilde{\Theta}}.
\]
\end{lemma}

To perform refinement, we identify the parameter location $ (u^*, v^*) \in [0,1]^2 $ at which the maximum deviation occurs:
\[
(u^*, v^*) = \arg\max_{i,j} \tilde{\delta}(\vec{P}_{i,j}),
\]
and insert new knots at these locations:
\[
\psi^* := u^*, \quad \theta^* := v^*, \quad \tilde{\Psi} := \Psi \cup \{\psi^*\}, \quad \tilde{\Theta} := \Theta \cup \{\theta^*\}.
\]

The updated control points $ \vec{P}_{i,j}' $ are computed recursively. For surfaces, refinement in both parametric directions yields:
\begin{equation}
\label{eq:refined_control_point}
\vec{P}_{i,j}' = (1 - \alpha_i)(1 - \beta_j) \vec{P}_{i-1,j-1} + (1 - \alpha_i)\beta_j \vec{P}_{i-1,j} + \alpha_i(1 - \beta_j) \vec{P}_{i,j-1} + \alpha_i \beta_j \vec{P}_{i,j},
\end{equation}
where
\[
\alpha_i = \frac{\psi^* - \psi_i}{\psi_{i+p} - \psi_i}, \quad \beta_j = \frac{\theta^* - \theta_j}{\theta_{j+q} - \theta_j}.
\]

The refined surface lies in the expanded space $ \mathbb{S}_{p,q,\tilde{\Psi},\tilde{\Theta}} $ and retains smoothness and continuity. The refinement procedure is applied iteratively until the deviation metric satisfies the following:
\[
\tilde{\epsilon}(\{\vec{P}_{i,j}\},  \vec{\mathcal{S}}) \leq \epsilon_{\mathrm{tol}}.
\]

This adaptive process ensures that the control points remain close to the surface they define, particularly in regions that require higher geometric accuracy. With this proximity, differential geometric quantities, such as unit normals and mean curvature, can be reliably evaluated using Equations~\eqref{eq:normalvector} and \eqref{eq:mean_curvature} at each data point. These quantities are then used in evolution equations to update the data points and control points in a geometry-aware manner.

\noindent
\textbf{Remark.}
Knot insertion enlarges the set of control points $ \{\vec{P}_{i,j}\} $ by augmenting the knot vectors $ \Psi $ and $ \Theta $ to $ \tilde{\Psi} = \Psi \cup \{\psi^*\} $ and $ \tilde{\Theta} = \Theta \cup \{\theta^*\} $, respectively. This increases the dimensionality of the control net and the spline space $ \mathbb{S}_{p,q,\tilde{\Psi},\tilde{\Theta}} $. To maintain a balanced and well-posed optimization in the subsequent evolution step, we introduce new data points $ \vec{x}_{\mathrm{new}} \in \mathbb{R}^3 $ by evaluating the refined surface $  \vec{\mathcal{S}}(u,v) $ at the Greville abscissae associated with the newly inserted knots: $\vec{x}_{\mathrm{new}} :=  \vec{\mathcal{S}}(u^{\mathrm{G}}, v^{\mathrm{G}})$. This ensures that the added degrees of freedom in the control net are guided by the geometry of the surface itself. The resulting point set remains representative of the surface, allowing the interpolation error (Definition~\ref{def:greville_deviation}) and the optimization procedure (Section~\ref{sec:po_evol_intererror}) to remain stable after refinement.
%%%%%%%%%%%%%%%%%%%%%%%%%%%%%%%%%%%%%%%%%%%%%%%%%%%%%%%%%%%%%%%%%
\subsection{Optimizing Control Points for Geometry-Driven Surface Evolution}
\label{sec:po_evol_intererror}

With a refined B-spline surface whose control points lie close to the surface, we perform geometry-driven evolution. This has two components: (i) update the discrete data points according to surface-derived velocities; (ii) optimize the control points to maintain an accurate surface representation.

Let $ \mathcal{P}_k \subseteq \mathbb{R}^3 $ denote the discrete data points in patch $ k $, with the core points $ \mathcal{P}_k^{\mathrm{core}} \subseteq \mathcal{P}_k $ used to drive the evolution. Each point $ \vec{x}_i(t) \in \mathcal{P}_k^{\mathrm{core}} $ evolves according to a velocity field derived from the geometric quantities of the B-spline surface representation. As introduced in Section~\ref{sec:general_framework}, the motion is governed by equation~\eqref{eq:moving_datavelocity}.

Since the B-spline surface $  \vec{\mathcal{S}}$ is defined by control points $ \{\vec{P}_{i,j}\}$, these control points must also evolve to reflect changes in geometry. We assume that each control point $ \vec{P}_{i,j}(t) $ moves with a velocity consistent with the surface:
\begin{equation}
\label{eq:moving_controlpoints}
    \vec{P}_{i,j}(t + \Delta t) = \vec{P}_{i,j}(t) + \Delta t \cdot \vec{V}_{i,j}(t),
\end{equation}
where $ \vec{V}_{i,j}(t) \in \mathbb{R}^3 $ is the velocity at the control point, typically defined by evaluating the surface velocity at the associated Greville abscissae $ (u_i^{\mathrm{G}}, v_j^{\mathrm{G}}) $.

After the evolution step, we assess the accuracy of the surface representation relative to the updated data points. The interpolation error is defined as:
\begin{equation}
\label{eq:interpolation_error}
    \mathrm{err}_{\mathrm{interp}} := \max_{\vec{x}_i \in \mathcal{P}_k} \left\| \vec{x}_i -  \vec{\mathcal{S}}(u_i, v_i) \right\|,
\end{equation}
where $ (u_i, v_i) $ denotes the parameter coordinates corresponding to $ \vec{x}_i $. If $ \mathrm{err}_{\mathrm{interp}} > \tau $, for a prescribed tolerance $ \tau > 0 $, the surface must be re-optimized to reduce this deviation.

To this end, we define the following least-squares optimization problem:
\begin{equation}
\label{eq:least_squares_min}
    \min_{\{\vec{P}_{i,j}\}} \sum_{\vec{x}_i \in \mathcal{P}_k} \left\| \vec{x}_i -  \vec{\mathcal{S}}(u_i, v_i) \right\|^2,
\end{equation}
which seeks the set of control points that best approximates the updated point-cloud in a global $ \ell^2 $-sense. We solve this problem iteratively using a Gauss-Seidel-type gradient descent. At each iteration $ k $, the control points are updated via:
\begin{equation}
\label{eq:gauss_seidel_update}
    \vec{P}_{i,j}^{(k+1)} = \vec{P}_{i,j}^{(k)} - \alpha \cdot \nabla_{\vec{P}_{i,j}} \left( \sum_{\vec{x}_\ell \in \mathcal{P}_k} \left\| \vec{x}_\ell -  \vec{\mathcal{S}}(u_\ell, v_\ell; \{\vec{P}_{i,j}^{(k)}\}) \right\|^2 \right).
\end{equation}
where $ \alpha > 0 $ is the step size and the gradient $ \nabla_{\vec{P}_{i,j}} $ is taken with respect to the control point $ \vec{P}_{i,j} $. The surface $  \vec{\mathcal{S}}(u,v; \{\vec{P}_{i,j}\}) $ is explicitly parameterized by the current control net, ensuring that optimization is performed on the surface geometry at each step. The iterations proceed until the interpolation error is sufficiently reduced or a maximum number of iterations is reached.

At each time step $ t \to t + \Delta t $, the data points and control points evolve according to the geometric velocity field, followed by an optimization of the control net to reduce the interpolation error. This process ensures that the B-spline surface remains smooth, accurate, and geometrically consistent with the evolving data. The error metric defined in Equation~\eqref{eq:interpolation_error} serves as a feedback mechanism, triggering the optimization of the control point whenever the surface deviation exceeds a prescribed tolerance.
%%%%%%%%%%%%%%%%%%%%%%%%%%%%%%%%%%%%%%%%%%%%%%%%%%%%%%%%%%%%%%%%%%%%%%%%%
\subsection{Optimizing point-cloud Density for Surface Evolution}
\label{sec:rem_ins_distr}
As the surface deforms, points in each core region $ \mathcal{P}_k^{\mathrm{core}} \subset \mathcal{X}(t) $ may become overly clustered or too sparse, degrading stability and accuracy. To address this, we introduce a local mechanism for inserting or removing data points within each core region $ \mathcal{P}_k^{\mathrm{core}} $, based on the sampling density in the parameter space.

Let the parameter pairs associated with  $ \mathcal{P}_k^{\mathrm{core}} $ be 
\[
\mathcal{U}_k := \{ (u_i, v_i) \in \Omega_k \subset [0,1]^2 \mid \vec{x}_i =  \vec{\mathcal{S}}(u_i, v_i),\ \vec{x}_i \in \mathcal{P}_k^{\mathrm{core}} \},
\]
where $ \Omega_k $ is the parameter subdomain corresponding to the $ k $-th core region. For each $ (u_i, v_i) \in \mathcal{U}_k $, let $ (u_j, v_j) \in \mathcal{U}_k $ be its nearest neighbor (in parameter or surface space), and define
\[
d_i := \left\|  \vec{\mathcal{S}}(u_i, v_i) -  \vec{\mathcal{S}}(u_j, v_j) \right\|.
\]

Given two thresholds $ d_{\min}, d_{\max} \in \mathbb{R}>0 $, we apply the following rules: Removal: If $ d_i < d_{\min} $, the point $ \vec{x}_i =  \vec{\mathcal{S}}(u_i, v_i) $ is removed from $ \mathcal{P}_k^{\mathrm{core}} $, and its corresponding parameter pair $(u_i, v_i)$ is removed from $ \mathcal{U}_k $. Insertion: If $ d_i > d_{\max} $, we insert a new parameter value $(u_{\text{new}}, v_{\text{new}}) \in \Omega_k$, chosen as a midpoint or subdivision between $(u_i, v_i)$ and $(u_j, v_j)$. The corresponding new data point is then calculated as $\vec{x}_{\text{new}} :=  \vec{\mathcal{S}}(u_{\text{new}}, v_{\text{new}}),$ and added to $\mathcal{P}_k^{\mathrm{core}} $.

This refinement is iterated within each core region until all pairwise distances satisfy
\[
d_{\min} \leq d_i \leq d_{\max}, \quad \forall \vec{x}_i \in \mathcal{P}_k^{\mathrm{core}}.
\]

 Let $ \mathcal{P}_k^{\mathrm{core}, (I)} $ be the core at iteration $I,$, and define, 
\[
\mathcal{P}_k^{\mathrm{core}, (I+1)} := \operatorname{Refine}(\mathcal{P}_k^{\mathrm{core}, (I)}, d_{\min}, d_{\max}).
\]

The refined data point set is then obtained by reassembling all core regions:
\begin{equation}
    \mathcal{X}(t)^{\mathrm{refined}} := \bigcup_{k=1}^K \mathcal{P}_k^{\mathrm{core}, (I_k)},
\end{equation}
where $ I_k $ is the final iteration index for the region $ k $. This process ensures that the point cloud remains geometrically well-distributed across the surface domain while preserving the non-overlapping structure of the core regions. It also guarantees that all inserted points lie exactly on the surface $  \vec{\mathcal{S}}(u,v) $ and that their parameter locations remain within the known subdomains $ \Omega_k $.

Because insertion/removal modifies $ \mathcal{X}(t) $, we recompute the control net $\{ \vec{P}_{i,j}(t) \}$ to interpolate the refined data, ensuring $  \vec{\mathcal{S}}(u,v; \{ \vec{P}_{i,j}(t) \}) $  remains consistent before proceeding to $ t \to t + \Delta t $.

%%%%%%%%%%%%%%%%%%%%%%%%%%%%%%%%%%%%%%%%%%%%%%%%%%%%%%%%%%%%%%%%%%%
\subsection{Computational Cost Analysis}
\label{sec:cost_analysis}

A principal goal of our framework is to reduce the computational cost associated with evolving surface point clouds. Traditional methods recompute surface interpolants at every time step. For $ N $ data points and stencil size $ m $, the interpolation cost per step is $ \mathcal{O}(N m^3) $, leading to a total cost over $ [0,T] $ of $\mathcal{O}\left( \frac{N m^3 T}{\Delta t} \right).$

In contrast, we construct an initial B-spline surface at $ t = 0 $ and then evolve it patch-wise using the decomposition introduced in Section~\ref{sec:patches}. Within each patch, the core points $ \mathcal{P}_k^{\mathrm{core}} $ drive the evolution, while the overlapping boundary points $ \mathcal{P}_k^{\mathrm{bdy}} $ ensure continuity. Let $ m_c $ denote the number of core points per patch, and $ K \sim N / m_c $ the total number of patches.

The initial interpolation requires solving $ K $ local linear systems, each of size $ m_c $. Since solving a system of size $ m_c $ has computational cost $ \mathcal{O}(m_c^3) $, the total one-time setup cost is $ K \times \mathcal{O}(m_c^3) = \mathcal{O}(K m_c^3) $. Substituting $ K \sim N / m_c $, we obtain an overall setup cost of $ \mathcal{O}(N m_c^2) $. During evolution, geometric quantities are updated at core points, and control points are adjusted via local optimization. This results in a per-step cost of $ \mathcal{O}(N m_c^2) $, and a total cost of $ \mathcal{O}\left( \frac{N m_c^2 T}{\Delta t} \right) $ over the interval $ [0, T] $. The overall cost of the proposed method is therefore $\mathcal{O}\left( N m_c^2 + \frac{N m_c^2 T}{\Delta t} \right)$, which is significantly lower than the traditional cost when $ m_c \ll m^3 $. The efficiency gain is captured by the ratio
\begin{equation}
    \frac{\text{Proposed}}{\text{Traditional}} \sim \frac{m_c^2}{m^3}.
\end{equation}
This reduction is achieved by reusing the initial surface structure and localizing control point updates, rather than recomputing the entire interpolant at each time step.

%%%%%%%%%%%%%%%%%%%%%%%%%%%%%%%%%%%%%%%%%%%%%%%
\section{Numerical Results and Simulations}\label{sec:numerical_results}

Before applying the proposed B-spline-based framework to simulate dynamic surface evolution, we first validate its precision for estimating geometric quantities on static surfaces. In particular, we assess the performance of the local B-spline fitting method in recovering unit normal vectors and mean curvature from scattered point-cloud data.

The test is carried out at the initial time $ t = 0 $ using synthetic point clouds sampled from a known smooth reference surface. We vary the total number of points in the cloud, with point counts $ N \in \{948, 1806, 2964, 3816, 4890\} $, to study the impact of resolution on accuracy. For each case, the surface is partitioned into disjoint core patches as described in Section~\ref{sec:patches}, and overlapping neighborhoods are constructed to allow local B-spline interpolation.

Figure~\ref{fig:error} presents errors in estimating unit normal vectors and mean curvature as functions of the number of core points $ m_c $ per patch. The errors are computed in the $ \ell^2 $-norm by comparing the estimated quantities with the exact analytical values defined on the reference surface. Each curve in the plots corresponds to a different point-cloud resolution. As expected, higher resolutions lead to improved accuracy. Interestingly, we observe that increasing the number of core points $ m_c $ tends to degrade the quality of the approximation. This behavior suggests that overly large patches may incorporate points from regions with varying curvature, leading to over-smoothing and reduced accuracy in local geometric estimation.

The choice of patch size is therefore critical for balancing efficiency and accuracy. A smaller number of core points $ m_c $ results in more localized B-spline fits, which better preserve fine-scale geometric features and reduce approximation error. However, working with smaller patches also increases the total number of local problems to solve, raising overall computational cost. Conversely, while larger patches may seem computationally attractive, they can negatively impact accuracy by blurring local geometric structure. This trade-off highlights the importance of selecting a patch size that ensures both computational feasibility and high-accuracy geometric reconstruction. The results in Figure~\ref{fig:error} underscore the method’s sensitivity to this choice.

We now demonstrate the method in dynamic scenarios, including mean curvature flow, curvature-driven transport, and biologically inspired tumor-growth dynamics on evolving geometries.
\begin{figure}
\begin{center}
\subfigure[Normal error]{
  \begin{overpic}[width=2.8in]{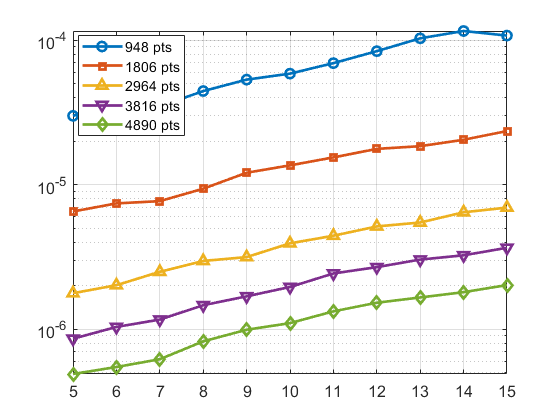}
    \put(30,300){\rotatebox{90}{\scriptsize $\ell^2$-norm error}}
    \put(500,-5){\scriptsize $m_c$}
  \end{overpic}
}
\subfigure[Mean curvature error]{
  \begin{overpic}[width=2.8in]{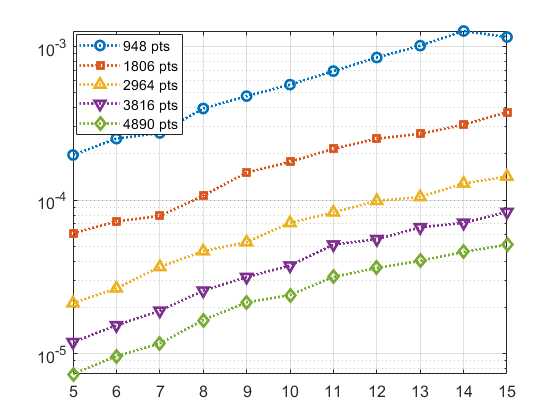}
    \put(30,300){\rotatebox{90}{\scriptsize $\ell^2$-norm error}}
    \put(500,-5){\scriptsize $m_c$}
  \end{overpic}
}
\caption{
$\ell^2$-norm error in estimating (a) unit normal vectors and (b) mean curvature at $ t = 0 $, using local B-spline fitting on scattered surface data. Each curve corresponds to a different point-cloud resolution with total points $ N \in \{948, 1806, 2964, 3816, 4890\} $. The x-axis represents the number of core points per patch $ m_c $, while the y-axis shows the corresponding $\ell^2$-norm error.
}
\label{fig:error}
\end{center}
\end{figure}

\subsection{Mean Curvature Flow}\label{sec:MCF}

This section demonstrates the application of the proposed B-spline-based geometric evolution method to simulate surface motion under mean curvature flow (MCF) using quasi-uniform point clouds. In the classical setting, MCF describes the deformation of a smooth hypersurface $\Gamma(t) \subset \mathbb{R}^3$ moving in the direction of its normal vector, with a speed proportional to its mean curvature. The evolution law is given by
\begin{equation}
    \frac{d \mathbf{x}}{dt} = -H \vec{n}, \label{eq:MCF_general}
\end{equation}
where $ H $ denotes the mean curvature (a scalar field) and $ \vec{n} $ is the unit outward normal vector at a point $ \mathbf{x} \in \Gamma(t) $. The vector $ -H \vec{n} $ is often referred to as the mean curvature vector, which encodes both the magnitude and direction of the surface motion.

In our framework, the surface is represented by a quasi-uniform point cloud. The quantities $ H $ and $ \vec{n} $ are estimated locally via B-spline fitting, and points are advected using a forward Euler discretization of \eqref{eq:MCF_general}.

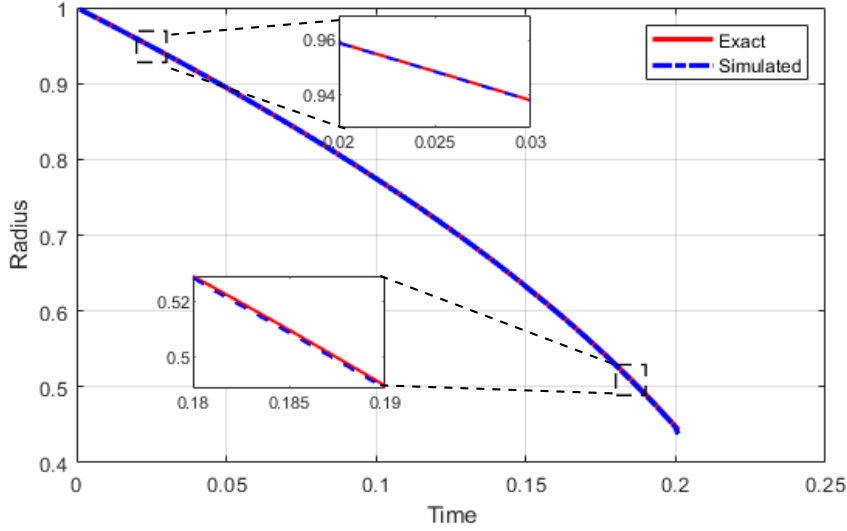
\begin{figure}[t]
\centering
\begin{tikzpicture}
    \node[anchor=south west,inner sep=0] (image) at (0,0) {\includegraphics[width=5in]{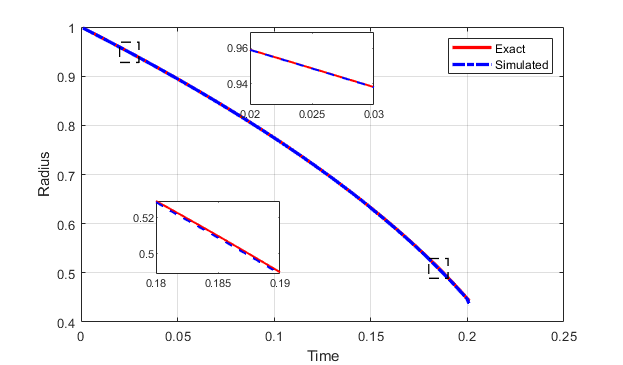}};
    \coordinate (A1) at (2.9,6.55);  
    \coordinate (B1) at (5.2,6.75);  
    \coordinate (C1) at (2.9,6.1);  
    \coordinate (D1) at (5.2,5.3);  
    \coordinate (A2) at (5.65,3.35);  
    \coordinate (B2) at (8.8,2.15);  
    \coordinate (C2) at (5.65,1.9);  
    \coordinate (D2) at (8.8,1.79); 
    \draw[dashed, black, thick] (A1) -- (B1); 
    \draw[dashed, black, thick] (C1) -- (D1); 
    \draw[dashed, black, thick] (A2) -- (B2); 
    \draw[dashed, black, thick] (C2) -- (D2); 
\end{tikzpicture}
\caption{Comparison between the numerically computed and analytic radius for the evolution of a sphere under mean curvature flow. The dashed blue curve shows the radius estimated at each time step by our point-cloud B-spline method, while the red line represents the exact solution given by Equation~\eqref{eq:sphere_exact}.}
\label{fig:radius_comparison}
\end{figure}
\begin{figure}
\begin{center}
\subfigure[$t=0$]{\includegraphics[width=1.5in]{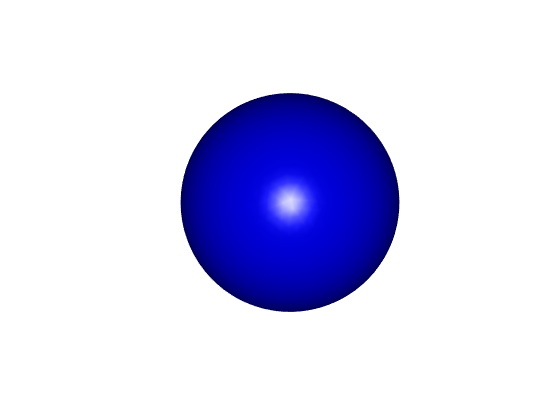}}\hspace{-25pt}
\subfigure[$t=0.05$]{\includegraphics[width=1.5in]{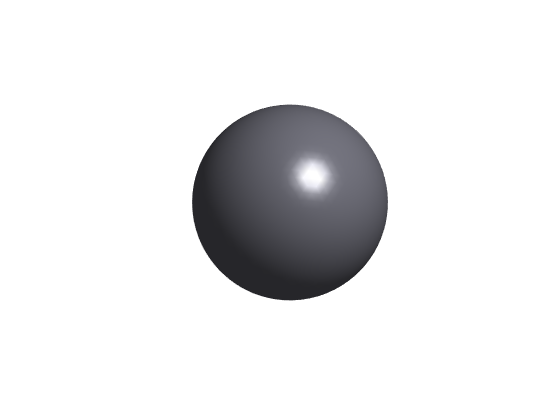}}\hspace{-25pt}
\subfigure[$t=0.10$]{\includegraphics[width=1.5in]{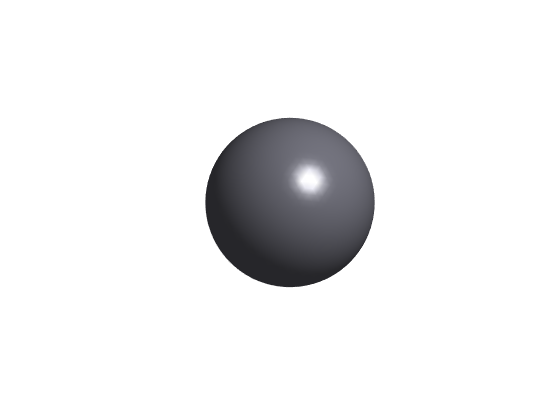}}\hspace{-25pt}
\subfigure[$t=0.15$]{\includegraphics[width=1.5in]{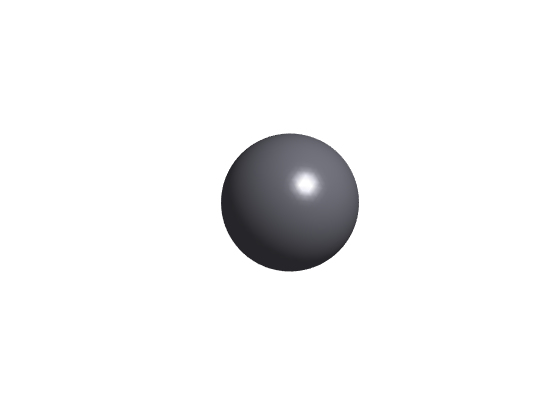}}\hspace{-25pt}
\subfigure[$t=0.20$]{\includegraphics[width=1.5in]{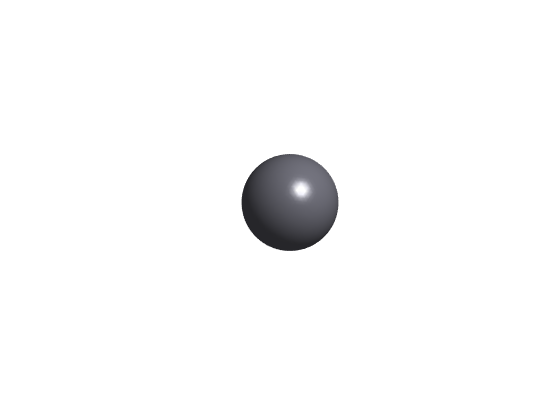}}\hspace{-15pt}
\caption{Evolution of a unit sphere under mean curvature flow, shown at representative times $t = 0$, $0.05$, $0.10$, $0.15$, and $0.20$. The point cloud shrinks isotropically as the surface contracts, illustrating the expected behavior for this classical test case.}
\label{fig:kmotiom_in_sphere}
\end{center}
\end{figure}

\subsection*{Example 1: Evolution of a Sphere}
As a preliminary test case, we consider a sphere with initial radius $ r(0) = 1 $. For a sphere, the mean curvature is spatially constant and equals $ H = 2/r $, so the evolution reduces to a purely radial contraction governed by the ordinary differential equation.
\begin{equation}
    \frac{dr(t)}{dt} = -\frac{2}{r(t)}, \label{eq:sphere_ODE}
\end{equation}
with a closed-form solution
\begin{equation}
    r(t) = \sqrt{r(0)^2 - 4t}. \label{eq:sphere_exact}
\end{equation}
The surface collapses at the extinction time $ t^* = r(0)^2 / 4 $.

To simulate this evolution, we generate a quasi-uniform point cloud of $ N = 4890 $ nodes on the unit sphere. At each time step, local B-spline surface patches are fitted to the neighborhood of each point, from which $ H $ and $ \vec{n} $ are computed. The points are then updated using the velocity field $ -H \vec{n} $, with time-step size $ \Delta t = 0.001 $. Figure~\ref{fig:kmotiom_in_sphere} illustrates the shrinking motion at different times, while Figure~\ref{fig:radius_comparison} compares the numerically estimated radius with the analytic solution \eqref{eq:sphere_exact}, showing excellent agreement.

\subsection*{Example 2: Evolution of an Ellipsoid}

To test the method on anisotropic geometries, we consider a point cloud sampled from an ellipsoid defined by the implicit equation
\begin{equation}
    \frac{x^2}{a^2} + \frac{y^2}{b^2} + \frac{z^2}{c^2} = 1, \label{eq:ellipsoid_general}
\end{equation}
where $ a, b, c > 0 $ are the semi-axes along the $ x $-, $ y $-, and $ z $-directions, respectively. In this experiment, we use $ a = 2 $, $ b = 1 $, and $ c = 1.5 $.

Equation \eqref{eq:ellipsoid_general} is used solely to construct the initial point cloud; the surface evolution is performed independently of any analytic description. A quasi-uniform set of $ N = 2842 $ points is sampled from the ellipsoid. As in the spherical case, local B-spline fitting is used to estimate curvature and normals, and the points are updated via \eqref{eq:MCF_general} with time step $ \Delta t = 0.001$.

As time progresses, the ellipsoid undergoes curvature-driven motion governed by mean curvature flow, causing the surface to shrink and become increasingly spherical before collapsing, akin to the behavior observed in the spherical case. This demonstrates the algorithm's ability to handle anisotropic geometries under curvature-driven flow. Figure~\ref{fig:mcf_ellipsoid} illustrates the evolution of the ellipsoid at different time steps, showing the progressive smoothing and shrinkage of the geometry under the mean curvature flow.
\begin{figure}
\begin{center}
\subfigure[$t=0$]{\includegraphics[width=1.5in]{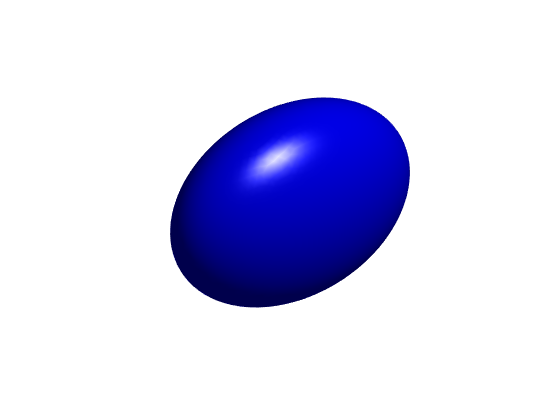}}\hspace{-25pt}
\subfigure[$t=0.2$]{\includegraphics[width=1.5in]{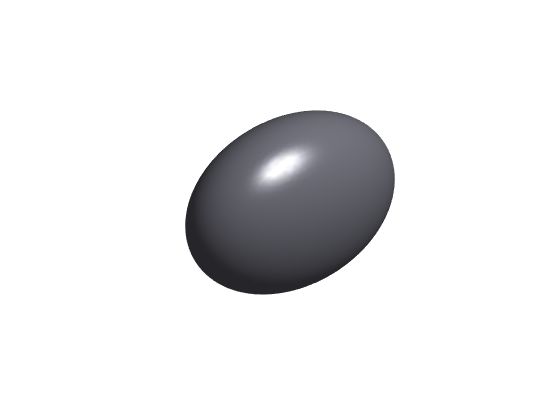}}\hspace{-25pt}
\subfigure[$t=0.4$]{\includegraphics[width=1.5in]{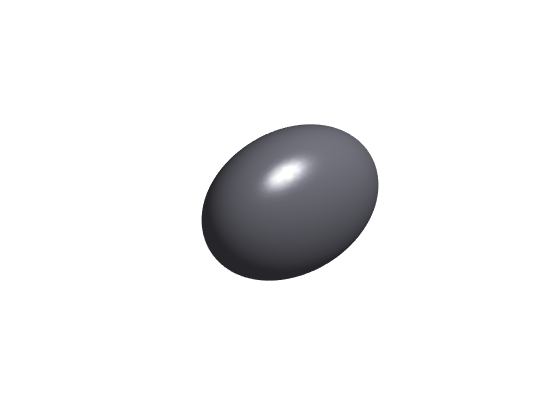}}\hspace{-25pt}
\subfigure[$t=0.6$]{\includegraphics[width=1.5in]{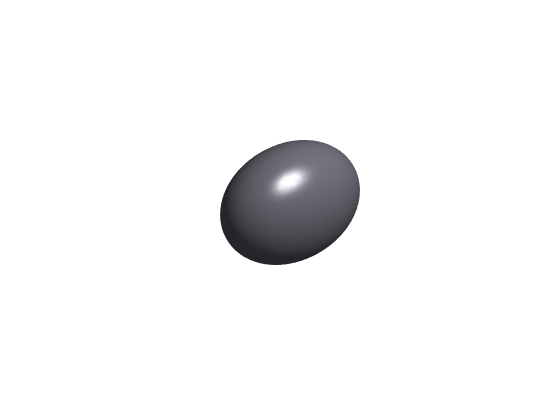}}\hspace{-25pt}
\subfigure[$t=0.8$]{\includegraphics[width=1.5in]{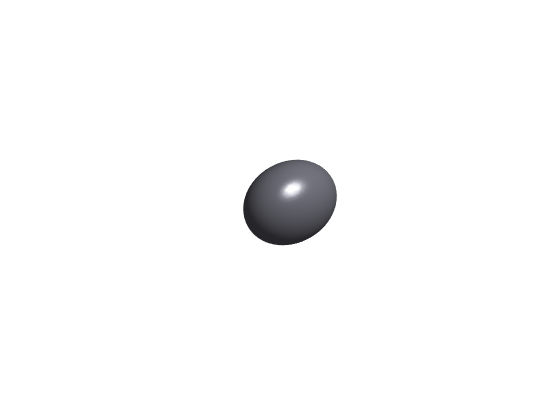}}\hspace{-15pt}
\caption{Evolution of a point-cloud sampled from an ellipsoid under mean curvature flow, shown at times $t = 0$, $0.2$, $0.4$, $0.6$, and $0.8$. The initial ellipsoid has semi-axes $a = 2$, $b = 1$, and $c = 1.5$. As the surface evolves, anisotropic shrinkage is observed: the ellipsoid becomes progressively rounder and eventually approaches a spherical shape before collapsing.}
\label{fig:mcf_ellipsoid}
\end{center}
\end{figure}

\subsection{Strongly Coupled Evolution on a Toroidal Surface}
We next consider a fully coupled problem involving both surface geometry and a scalar field defined on the surface. The geometry evolves in response to curvature and the scalar field, while the field itself evolves on the moving surface. Such interactions are common in physical and biological systems. This setting provides a rigorous test of the proposed numerical framework.

The initial surface is a torus with major radius $ R = 1 $ and minor radius $ r = 0.3 $. A quasi-uniform point cloud is generated by sampling the toroidal surface defined by the equation
\begin{equation}
    \Phi(x, y, z) := \left( R - \sqrt{x^2 + y^2} \right)^2 + z^2 - r^2 = 0,
    \label{eq:torus_points}
\end{equation}
which serves only as a geometric reference for initialization. The subsequent evolution is performed entirely within the point-cloud framework, using the B-spline-based method to estimate geometric quantities and surface differential operators, with no use of implicit surface representations.

The interface evolves according to a curvature- and field-dependent velocity law,
\begin{equation}
    \vec{V} = (\varepsilon H + \delta u)\vec{n},
    \label{eq:torus_velocity_general}
\end{equation}
where $ H $ is the mean curvature, $ \vec{n} $ is the outward unit normal vector, and $ \varepsilon, \delta > 0 $ are model parameters. The scalar field $ u \colon \Gamma(t) \rightarrow \mathbb{R} $ represents a time-dependent quantity defined on the surface, which may be prescribed analytically, obtained from experimental data, or computed externally (for example, by solving a reaction-diffusion equation on the evolving surface).

In many applications, the scalar field $ u $ may itself evolve according to a transport-diffusion equation on the moving surface, such as
\begin{equation}
    \frac{\partial u}{\partial t} + \vec{V} \cdot \nabla_\Gamma u - u \nabla_\Gamma \cdot \vec{V} - \Delta_\Gamma u = 0,
    \label{eq:scalar_transport_on_surface}
\end{equation}
where $ \nabla_\Gamma $ and $ \Delta_\Gamma $ denote the surface gradient and Laplace-Beltrami operators, respectively. This equation models the advection, surface diffusion, and geometric effects on $ u $ as the surface evolves.  
However, in the present work, we do not solve this PDE within our framework; instead, we assume that the value of $ u $ is available at each time step, regardless of how it is obtained. This allows the method to flexibly accommodate a wide range of physical and biological scenarios.

The velocity field $ \vec{V} $ is then used to update the positions of the point cloud, driving the geometric evolution of the surface. The B-spline-based method provides local interpolation and estimation of the necessary geometric quantities at each step.

This setup leads to a strongly coupled geometric evolution problem in which the scalar field and geometry interact in non-linear ways. Initially, large spatial gradients in $ u $ induce strong anisotropic deformation of the torus, while surface diffusion (if present in the external computation of $ u $) acts to smooth the field, leading to stabilization of the geometry. Figure~\ref{fig:kmotion_in_torus} illustrates the evolution of the toroidal surface at several time instances.

\begin{figure}
\centering
\begin{overpic}[width=6in]{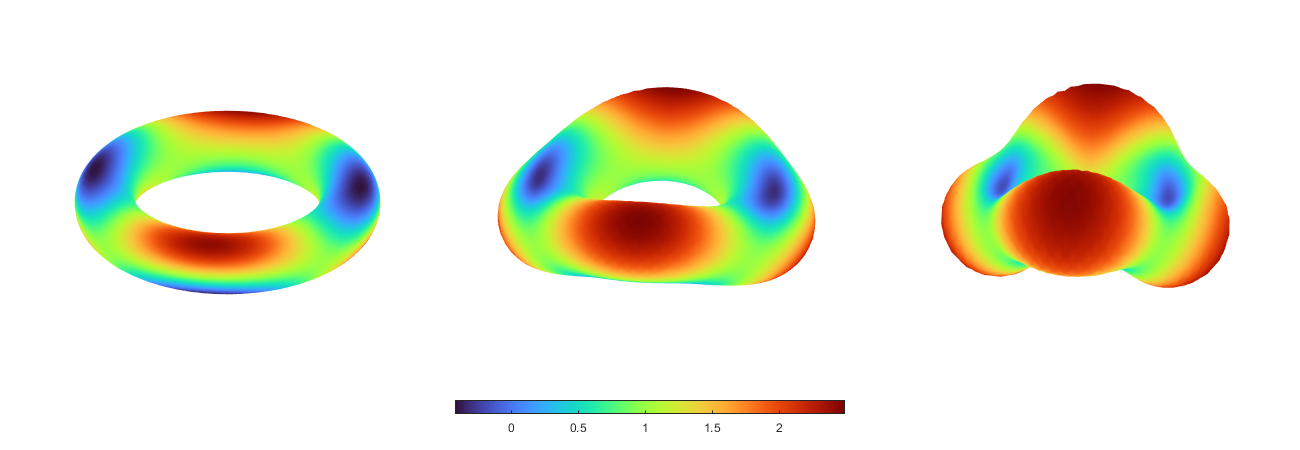}
  \put(150,300){\small $t = 0$}
  \put(450,300){\small $t = 0.25$}
  \put(800,300){\small $t = 0.5$}
\end{overpic}
\caption{Evolution of a toroidal surface and scalar field at $t = 0$, $0.25$, and $0.5$ under a coupled curvature- and field-driven velocity law. The snapshots show how strong initial gradients in the field induce anisotropic deformation, which is gradually smoothed by surface diffusion. This example demonstrates the algorithm's ability to simulate complex surface–field interactions.}
\label{fig:kmotion_in_torus}
\end{figure}

\subsection{A Theoretical Model for Tumor Growth on an Evolving Surface}
To further demonstrate the flexibility of the proposed method, we consider a canonical model for early-stage avascular tumor growth, in which the evolution of the tissue interface is coupled to a pair of reaction-diffusion equations on the deforming surface. This model is adapted from~\cite{petras2019least} and serves as a benchmark for strongly coupled surface–field dynamics.

Let $ u, w \colon \Gamma(t) \rightarrow \mathbb{R} $ denote the concentrations of a growth-promoting factor and a growth-inhibiting factor, respectively. The evolution of these fields is typically governed by the reaction-diffusion system
\begin{equation}
\begin{aligned}
\frac{D u}{D t} &= \Delta_\Gamma u - u \nabla_\Gamma \cdot \vec{V} + f_1(u, w), \\
\frac{D w}{D t} &= \mathcal{D} \Delta_\Gamma w - w \nabla_\Gamma \cdot \vec{V} + f_2(u, w),
\end{aligned}
\label{eq:tumor_model_system}
\end{equation}
where $ \frac{D}{Dt} $ is the material derivative, $ \Delta_\Gamma $ is the Laplace–Beltrami operator, and $ f_1, f_2 $ are nonlinear source terms.

While we reference this coupled PDE model to illustrate the kinds of multiphysics problems relevant to evolving biological interfaces, it is important to emphasize that the focus of this work is not on solving the reaction-diffusion system itself. Instead, the tumor growth example serves to demonstrate that the proposed point-cloud framework can accommodate arbitrary scalar fields, regardless of how they are obtained, and efficiently couple them to surface evolution. This highlights the method’s versatility and its potential for integration with a wide range of surface–field models.

The interface velocity is given by the law defined in~\eqref{eq:torus_velocity_general}, where $\vec{V}$ depends on both mean curvature and the scalar field $u$. The evolution of the point-cloud surface is thus directly influenced by these fields, irrespective of their origin. For the simulations, the values of $u$ and $w$ are provided at each time step. All geometric and differential quantities are estimated locally on the point cloud using the B-spline-based methods described in Section~\ref{sec:adaptive_bspline_surface}.

For the experiments shown in Figures~\ref{fig:kmotiom_TG2}, the parameters $\delta = 0.4$, $\varepsilon = 0.01$, number of points $N = 4890$, and time step size $\Delta t = 0.001$ are used. This example highlights the capacity of the proposed method to handle complex multiphysics systems on evolving surfaces, in which both the geometry and the fields evolve in a strongly coupled manner.
\begin{figure}
\centering
\begin{overpic}[width=6in]{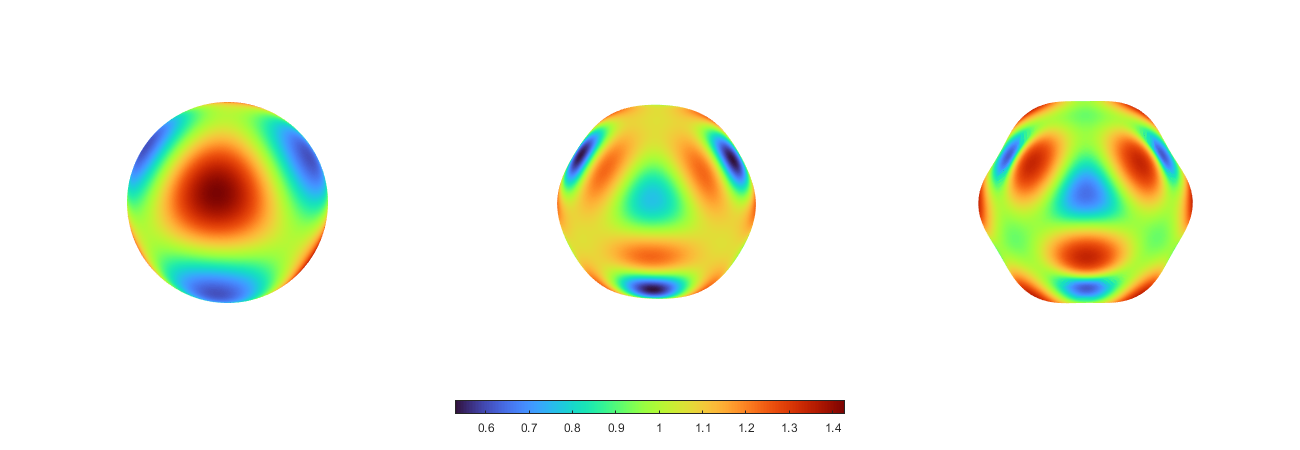}
  \put(150,300){\small $t = 0$}
  \put(450,300){\small $t = 0.025$}
  \put(790,300){\small $t = 0.05$}
\end{overpic}
\caption{Evolution of a spherical tumor interface and associated chemical fields at different times using a coupled reaction-diffusion and surface evolution model. The results demonstrate how the interface and the concentrations of growth-promoting and growth-inhibiting factors co-evolve on an adaptive point cloud, highlighting the algorithm’s capability to simulate complex multiphysics processes on evolving surfaces.}
\label{fig:kmotiom_TG2}
\end{figure}

\section{Conclusion}\label{sec:conclusion}
We presented a Lagrangian framework for evolving point-cloud surfaces using locally supported tensor-product B-spline patches. The key idea is to avoid global re-interpolation at every time step: control points are advanced together with data points, and light local iterations are used to maintain interpolation accuracy. This design preserves a consistent, differentiable surface representation while substantially reducing computational cost and avoiding global meshing or global re-parameterization. To support this workflow, we adopt a conditioning-aware local interpolation setup and a Greville-based deviation metric that triggers modest adjustments, such as knot insertion and point redistribution, only when needed. These auxiliary components help keep local fits accurate but do not alter the overall structure of the method.

Numerical experiments demonstrate the ability of the framework to accurately estimate differential geometric quantities and simulate a range of surface evolution problems, including mean curvature flow and strongly coupled surface field dynamics. For classical flows, the method replicates the theoretical behavior of shrinking geometries such as spheres and ellipsoids, while for coupled systems, it captures the nonlinear feedback between scalar fields and surface deformation, as exemplified by tumor growth and reaction-diffusion dynamics on evolving geometries. The computational cost analysis confirms that the localized patch construction and control point reuse yield significant reductions in per-step complexity compared to traditional global interpolation methods, with observed scaling consistent with theoretical bounds.

These results and the accompanying cost analysis indicate practical efficiency across the tested scenarios. At the same time, the present framework assumes a single smooth codimension-one surface and does not address topological changes (e.g., merging or splitting); oriented point data available for $t > 0$ could be used to detect opposing approaches and enable topology updates. Our experiments focus on closed, quasi-uniformly sampled surfaces, leaving highly non-uniform or noisy data and sharp-featured geometries for future assessment. Furthermore, coupling this framework with the method of fundamental solutions offers a promising direction for meshless arbitrary Lagrangian--Eulerian approaches to surface evolution~\cite{chen2021improved,fan2019localized}, enabling flexible, boundary-fitted simulations of evolving manifolds without global meshing.

\section*{Acknowledgment}
This work was supported by the General Research Fund (GRF No. 12301824, 12300922) of Hong Kong Research Grant Council.

  \bibliographystyle{elsarticle-num}
  \bibliography{ref-abrv}
\end{document}